\documentclass[11pt,reqno, final]{amsart}

\usepackage{color}
\usepackage[colorlinks=true, allcolors=blue]{hyperref}
\usepackage{amsmath, amssymb, amsthm}
\usepackage{mathrsfs}
\usepackage{mathtools}
\usepackage[noabbrev,capitalize,nameinlink]{cleveref}
\crefname{equation}{}{}
\usepackage{fullpage}
\usepackage{graphics}
\usepackage{pifont}
\usepackage{tikz}
\usepackage{bbm}
\usepackage[T1]{fontenc}

\usetikzlibrary{arrows.meta}

\usepackage{environ}
\usepackage{framed}
\usepackage{url}
\usepackage[linesnumbered,ruled,vlined]{algorithm2e}
\usepackage[noend]{algpseudocode}
\usepackage[labelfont=bf]{caption}
\usepackage{cite}
\usepackage{framed}
\usepackage[framemethod=tikz]{mdframed}
\usepackage{appendix}
\usepackage{graphicx}
\usepackage[textsize=tiny]{todonotes}
\usepackage{tcolorbox}
\allowdisplaybreaks[1]

\crefname{algocf}{Algorithm}{Algorithms}

\crefname{equation}{}{} 
\AtBeginEnvironment{appendices}{\crefalias{section}{appendix}} 

\usepackage[color]{showkeys} 

\colorlet{refkey}{orange!20}
\colorlet{labelkey}{blue!30}

\crefname{algocf}{Algorithm}{Algorithms}

\numberwithin{equation}{section}
\newtheorem{theorem}{Theorem}[section]

\newtheorem{lemma}[theorem]{Lemma}

\crefname{claim}{Claim}{Claims}

\newtheorem*{question*}{Question}

\theoremstyle{definition}
\newtheorem{definition}[theorem]{Definition}

\newtheorem*{definition*}{Definition}
\newtheorem{example}[theorem]{Example}

\theoremstyle{remark}
\newtheorem*{remark}{Remark}

\newcommand{\norm}[1]{\left\lVert#1\right\rVert}
\newcommand{\snorm}[1]{\lVert#1\rVert}

\newcommand{\sang}[1]{\langle #1 \rangle}

\newcommand{\mb}{\mathbb}

\newcommand{\mc}{\mathcal}

\newcommand{\on}{\operatorname}

\allowdisplaybreaks

\title{On the smoothed analysis of the smallest singular value with discrete noise}

\author[A1]{Vishesh Jain}
\address{Simons Institute for the Theory of Computing,
Berkeley, CA 94720, USA}
\email{visheshj@stanford.edu}

\author[A2]{Ashwin Sah}
\author[A3]{Mehtaab Sawhney}
\address{Department of Mathematics, Massachusetts Institute of Technology, Cambridge, MA 02139, USA}
\email{\{asah,msawhney\}@mit.edu}

\begin{document}
\begin{abstract}
Let $A$ be an $n\times n$ real matrix, and let $M$ be an $n\times n$ random matrix whose entries are i.i.d sub-Gaussian random variables with mean $0$ and variance $1$. We make two contributions to the study of $s_n(A+M)$, the smallest singular value of $A+M$.
\begin{enumerate}
    \item We show that for all $\epsilon \ge 0$,
    $$\mb{P}[s_n(A + M) \le \epsilon] = O(\epsilon \sqrt{n}) + 2e^{-\Omega(n)},$$
    provided only that $A$ has $\Omega (n)$ singular values which are $O(\sqrt{n})$. This extends a well-known result of Rudelson and Vershynin, which requires all singular values of $A$ to be  $O(\sqrt{n})$. 
    \item We show that any bound of the form
    $$\sup_{\snorm{A}\le n^{C_1}}\mb{P}[s_n(A+M)\le n^{-C_3}] \le n^{-C_2}$$
    must have $C_3 = \Omega (C_1 \sqrt{C_2})$. This complements a result of Tao and Vu, who proved such a bound with $C_3 = O(C_1C_2 + C_1 + 1)$, and counters their speculation of possibly taking $C_3 = O(C_1 + C_2)$.  
\end{enumerate}
\end{abstract}
\maketitle
\section{Introduction}\label{sec:introduction}
Recall that for an $n\times n$ matrix $A$, its singular values $s_1(A)\ge \dots \ge s_n(A)$ are defined to be the eigenvalues of $\sqrt{A^{T}A}$ arranged in non-decreasing order. Recall also that the largest and smallest singular values admit the following characterization:
\begin{align*}
    s_1(A) &= \sup_{x\in \mb{S}^{n-1}}\snorm{Ax}_{2} = \snorm{A};\\
    s_n(A) &= \inf_{x\in \mb{S}^{n-1}}\snorm{Ax}_2 = \snorm{A^{-1}}^{-1},
\end{align*}
where $\mb{S}^{n-1}$ denotes the unit sphere in Euclidean space $\mb{R}^{n}$, $\snorm{\cdot}_{2}$ denotes the Euclidean norm in $\mb{R}^{n}$, and $\snorm{\cdot}$ denotes the spectral norm of the matrix.  

The extreme singular values $s_1(A), s_n(A)$, and the condition number $$\kappa(A) := s_1(A)/s_n(A) = \snorm{A}\cdot \snorm{A}^{-1}$$
are especially important in applications. In particular, consider the fundamental problem of solving the linear system $Ax = b$, where $A\in \mb{R}^{n\times n}$ and $b\in \mb{R}^{n}$ are given. Suppose that there is an error $\delta b$ in inputting $b$, which leads to an error $\delta x$ in the output. Then,
\begin{align*}
    \frac{\snorm{\delta x}_{2}/\snorm{x}_2}{\snorm{\delta b}_2/\snorm{b}_2} = \left(\frac{\snorm{A^{-1}(\delta b)}_2}{\snorm{\delta b}_{2}}\right)\left(\frac{\snorm{Ax}_2}{\snorm{x}_2}\right) \le \kappa(A),
\end{align*}
which motivates defining the \emph{loss of precision} \cite{smale1985efficiency} as
$L(A) := \log \kappa(A).$

\subsection{Smoothed analysis}
In their celebrated work on \emph{smoothed analysis} \cite{spielman2004smoothed}, Spielman and Teng sought to understand why algorithms with poor worst-case performance are successful in practice. In the context of solving the linear system $Ax = b$, their program amounts to the following: even if the desired input matrix $A$ is ill-conditioned (i.e, has a large condition number, and hence, large loss of precision), it is likely that the computer will actually work with some small perturbation $A + M$, where $M$ represents the effect of random noise. Therefore, if we can show that for any fixed $A \in \mb{R}^{n\times n}$, the random matrix $A + M$ is well-conditioned (i.e., has small condition number) with high probability, this would help explain why the large condition number of $A$ does not lead to correspondingly bad performance in practice. 

Since the operator norm of a random matrix $M$ with independent entries (with sufficiently light tails) can be controlled using standard concentration techniques (cf. \cite{bai2010spectral,vershynin2018high}) (this immediately implies control on the operator norm of $A+M$ using the triangle inequality), it follows that they key challenge in controlling the condition number of $A+M$ is in understanding the behavior of the smallest singular value of $A+M$. In this direction, Sankar, Spielman, and Teng \cite{sankar2006smoothed} considered the case when the entries of $M$ are i.i.d Gaussian random variables with mean $0$ and variance $1$, and showed that for any $A \in \mb{R}^{n\times n}$ and for any $\epsilon \ge 0$,
\begin{align}
\label{eqn:SST}
    \mb{P}[s_n(A+M)\le \epsilon] \le C\epsilon \sqrt{n},
\end{align}
for $C\sim 2.35$ (recently, the constant $C$ was improved by Banks et al. to the optimal value of $1$ in \cite{banks2019gaussian}). This result is almost best possible, in the sense that when $A = 0$, it was shown by Edelman \cite{edelman1988eigenvalues} that for sufficiently small $\epsilon \ge 0$,
$$\mb{P}[s_n(M) \le \epsilon] \ge c\epsilon \sqrt{n},$$
where $c > 0$ is an absolute constant (in fact, $c$ can be taken arbitrarily close to $1$ by further restricting the range of $\epsilon$). 

\subsection{Smallest singular value of random matrices}
The proof of Sankar, Spielman, and Teng heavily exploited the orthogonal invariance of the Gaussian distribution, and extending their result to other distributions, especially discrete distributions, is very challenging, and has attracted much attention in the past fifteen years. Notably, in a seminal work, Rudelson and Vershynin \cite{RV08} showed that for a random matrix $M$ whose entries are i.i.d copies of a sub-Gaussian random variable with variance $1$ and mean $0$ (recall that a random variable $X$ is said to be sub-Gaussian if its sub-Gaussian norm, defined by $\snorm{X}_{\psi_{2}}:=\inf\{s > 0: \mb{E}[\exp(X^{2}/s^{2})]\le 2\}$, is finite), 
\begin{align}
\label{eqn:RV}
    \mb{P}[s_n(M) \le \epsilon] \le C\epsilon \sqrt{n} + 2\exp(-cn),
\end{align}
 where $C,c > 0$ depend only on the sub-Gaussian norm of the distribution. We note that this result is almost best-possible; the first summand has the optimal dependence on $\epsilon$ and $n$ as explained before, whereas the second summand is necessary to account for the probability that the random matrix is singular (for instance, if each entry of $M$ is independently $\pm 1$ with probability $1/2$ each, then any two given rows are equal with probability $2^{-n}$).  
 
 Since the appearance of \cite{RV08}, considerable work has been devoted to relaxing the assumptions in the work of Rudelson and Vershynin, and establishing a bound of the form \cref{eqn:RV} in greater generality. For instance, it was very recently shown by Livshyts, Tikhomirov, and Vershynin \cite{livshyts2019smallest} (see also the references there for previous work) that \cref{eqn:RV} holds under only the assumption that the entries of $M$ are independent, have mean $0$, have uniformly bounded variance, and are uniformly anti-concentrated (in fact, their result is a bit more general, as we will discuss later). In the setting of smoothed analysis, Tikhomirov \cite{tikhomirov2020invertibility} showed that for any $A \in \mb{R}^{n\times n}$ and for all $\epsilon > 0$,
 \begin{align*}
     \mb{P}[s_n(A+M)\le \epsilon] \le C\epsilon \sqrt{n},
 \end{align*}
if the rows of $M$ are independent \emph{continuous} random vectors satisfying certain technical assumptions (in particular, his result allows the rows of $M$ to be independent centered log-concave isotropic random vectors). We mention that when the entries of $M$ are independent continuous random variables with uniformly bounded densities, a bound of the form \cref{eqn:SST}, with the optimal $\sqrt{n}$ replaced by the sub-optimal $n$, is rather easy to prove (cf. \cite{tikhomirov2020invertibility}); for \emph{symmetric} matrices $M$ whose distributions on and above the diagonal are independent continuous random variables with uniformly bounded densities and \emph{symmetric} matrices $A$, a bound of the form \cref{eqn:SST} was obtained by Farrell and Vershynin \cite{farrell2016smoothed}, except that $\sqrt{n}$ is replaced by the sub-optimal $n^{2}$ (in the special case when $M$ is drawn from the appropriately normalised Gaussian Orthogonal Ensemble, the optimal $\sqrt{n}$ dependence was obtained by Bourgain \cite{bourgain2017problem}).       

\subsection{Smoothed analysis with discrete noise}
Since the noise/randomness encountered in practice (e.g., in digital devices) is typically discrete, it is natural to try to understand analogues of \cref{eqn:SST} when the distribution of the entries of $M$ is allowed to have atoms. This was first considered by Tao and Vu \cite{tao2008random}, who showed that if the entries of $M$ are i.i.d copies of a (possibly complex) random variable with mean $0$ and variance $1$, then for any constants $C_1, C_2 > 0$, there exists a $C_3 > 0$ such that 
\begin{align}
\label{eqn:TV}
\sup_{\snorm{A} \le n^{C_1}}\mb{P}[s_n(A+M) \le n^{-C_3}] \le n^{-C_2}.
\end{align}
Recently, it was shown by Jain \cite{jain2019quantitative} that a bound of this nature continues to hold under the more general condition $\snorm{A} \le \exp(n^{c})$, and for all target probabilities larger than $\exp(-n^{c'})$, where $c,c' > 0$ are small absolute constants. The best-known dependence of $C_3$ on $C_1, C_2$ was obtained in another work by Tao and Vu \cite{TV10}, who showed that
\begin{align}
    \label{eqn:TV-dependence}
    \sup_{\snorm{A} \le n^{C_1}}\mb{P}[s_n(A+M) \le n^{-2(C_2 + 2)C_1 + 1/2}] \le O(n^{-C_2 + o(1)} + \mb{P}[\snorm{M} \ge n^{C_1}]).
\end{align}

Perhaps surprisingly, it turns out that the dependence of the smallest singular value on $C_{1}$ is not an artifact of the proof, and is necessary in general, as the following unpublished example due to Rudelson shows (a similar example appears in independent work of Tao and Vu \cite{TV10}). 

\begin{example}
\label{example:rudelson}
Let $k \in [n]$, and let $A = \on{diag}(L,L,\dots,L,0,0,\dots,0)$, where the first $n-k$ entries are $L$, and the last $k$ entries are $0$. Let the entries of $M$ be independent lazy Rademacher random variables (i.e., each entry independently takes on the value $0$ with probability $1/2$, $-1$ with probability $1/4$, and $1$ with probability $1/4$). Let $(v_1,\dots,v_n) \in \mb{S}^{n-1}$ be orthogonal to the first $n-1$ rows of $A+M$. It is easily seen that $\sqrt{v_{1}^{2}+\dots+v_{n-k}^{2}} = O(\sqrt{n}/L)$. Then, since the last $k$ coordinates of the last row are $0$ with probability $2^{-k}$, it follows that with probability at least $2^{-k}$, the smallest singular value of $A+M$ is at most $O(\sqrt{n}/L)$.  
\end{example}

\subsection{Our results}
Note that \cref{example:rudelson} rules out the possibility of a bound of the form \cref{eqn:RV} for general $A+M$, even when the entries of $M$ are i.i.d sub-Gaussian random variables. On the other hand, such a bound is known to hold (restricting ourselves to the case when $M$ has i.i.d sub-Gaussian entries with mean $0$ and variance $1$) if $\snorm{A} = O(\sqrt{n})$ \cite{RV08}, and more generally, if $\sum_{i=1}^{n}s_i(A)^{2} = O(n^{2})$ \cite{livshyts2019smallest}.

Our first main result shows that when the entries of $M$ are i.i.d sub-Gaussian random variables with mean $0$ and variance $1$, \cref{eqn:RV} continues to hold for a much wider class of $A+M$. In order to state the result, we need the following definition. 

\begin{definition}\label{def:stable-rank}
Let $m\le n \in \mb{N}$ and $K \ge 0$. 
An $m \times n$ matrix $A$ has $K$-rank $r$ if it has exactly $r$ singular values of size at least $K\sqrt{n}$. 
\end{definition}
\begin{theorem}\label{thm:low-stable-rank}
Let $K > 0$ and $\eta \in (0,1)$. Let $\xi$ be a sub-Gaussian random variable with mean $0$ and variance $1$, and let $M$ be an $n\times n$ random matrix, each of whose entries is an independent copy of $\xi$. Then, for any $A \in \mb{R}^{n\times n}$ with $K$-rank at most $(1-\eta)n$, we have
\begin{align}
\label{eqn:our-bound}
\mb{P}[s_n(A+M)\le\epsilon]\le C\epsilon\sqrt{n} + 2e^{-cn},
\end{align}
where $C, c > 0$ depend only on $K, \eta$, and the sub-Gaussian norm of $\xi$.
\end{theorem}
\begin{remark}
~~~~~
\begin{enumerate}
    \item In the setting where  $M$ has i.i.d sub-Gaussian entries with mean $0$ and variance $1$,
    the condition that the $K$-rank of $A$ is at most $(1-\eta)n$ is a relaxation of the condition $\sum_{i=1}^{n} s_i(A)^{2} = O(n^{2})$ appearing in \cite{livshyts2019smallest} (although, note that the noise $M$ in \cite{livshyts2019smallest} may be inhomogeneous, and is allowed to have much heavier tails). Notably, it encompasses and generalizes the class of low-rank matrices, which are especially important in applications (see, e.g., \cite{abbe2020entrywise, o2018random}).  
    \item The dependence of $c$ on $\eta$ is necessary in general. Indeed,  \cref{example:rudelson} provides an example of $A$ and $M$ such that $A$ has $0$-rank $n-k$, and such that on an event of probability at least $2^{-k}$, $s_n(A+M)$ depends on $\snorm{A}$.  
    \item The only place where sub-Gaussianity of $\xi$ is used in our proof is to guarantee that $\snorm{M} = O(\sqrt{n})$ except with exponentially small probability. It is well-known \cite{bai2010spectral} that $\snorm{M} = O(\sqrt{n})$ with high probability, as long as the entries of $M$ have uniformly bounded fourth moment. Consequently, \cref{thm:low-stable-rank} continues to hold under the more general assumption that $\xi$ is a random variable with mean $0$, variance $1$, and finite fourth moment, provided we add the term $(1-\mb{P}[\snorm{M} = O(\sqrt{n})])$ to the right-hand side of \cref{eqn:our-bound}. 
    \item In \cite{tikhomirov2016smallest}, it was shown by Tikhomirov that if $M'$ is an $N \times n$ matrix ($N \ge (1+\delta)n$) with i.i.d uniformly anti-concentrated entries, then for any $N \times n$ matrix $A'$, $\mb{P}[s_n(A'+M') \le u\sqrt{N}] \le 2\exp(-vN)$, where $u,v$ depend only on $\delta > 0$ and the uniform bound on the anti-concentration. Compared to \cite{tikhomirov2016smallest}, the main innovation in our work (\cref{lem:metric-entropy-LCD}) is the consideration of the arithmetric structure of normal vectors to random hyperplanes without any \emph{a priori} control on the operator norm (this is unnecessary in the rectangular case).     
\end{enumerate}
\end{remark}

Next, we study the influence of $A$ on the lower tail of $s_n(A+M)$. As noted above, in the case when $\snorm{A} \le n^{C_1}$ and the target probability is $n^{-C_2}$, the best-known lower bound on the smallest singular value  due to Tao and Vu \cref{eqn:TV-dependence} is of the form $n^{-O(C_1C_2 + C_1 + 1)}$. It is natural to ask whether this dependence may be improved, in particular, if the term $C_1C_2$ may be replaced by $C_1 + C_2$ (this seems to have been speculated by Tao and Vu in \cite{TV10}). Our second main result shows that a dependence of the form $O(C_1 + C_2)$ is not possible, and in fact, that the term $C_1 C_2$ cannot be replaced by anything asymptotically smaller than $C_1 \sqrt{C_2}$. 

\begin{theorem}\label{thm:counterexample}
Let $R_n$ be an $n\times n$ random matrix, each of whose entries is an independent copy of a lazy Rademacher random variable (i.e., $0$ with probability $1/2$, and $\pm 1$ with probability $1/4$ each). There exist positive constants $K, C > 0$ for which the following holds. Let $L\ge 2K\sqrt{n}$, and  let $A = \on{diag}(L,\ldots,L,0)$. Then, for each positive integer $t$ and for all $n$ sufficiently large (depending only on $t$), we have
\[\mb{P}\bigg[s_n(A+R_n)\le C\bigg(\frac{K\sqrt{n}}{L}\bigg)^t\bigg]\ge(4C)^{-t}K^{-\frac{(t-1)(t-2)}{2}}(\log 2t)^{1-t}n^{-\frac{t(t-1)}{4}}.\]
\end{theorem}
\begin{remark}
The dependence of the probability on $n$ should be sharp. However, establishing this would essentially require a joint local central limit theorem involving certain low-degree polynomials. Moreover, such a result would likely allow extension to the case when the entries of $R_{n}$ are Rademacher random variables (as we will see, our proof is considerably simplified by using the $2$-divisibility of lazy Rademacher random variables).
\end{remark}

\subsection{Acknowledgements}
We thank Yang P. Liu for discussions related to the relevant convex geometry and Mark Rudelson for valuable comments on the manuscript. 

\section{Proof of \texorpdfstring{\cref{thm:low-stable-rank}}{Theorem 1.2}}\label{sec:proof}
The proof of \cref{thm:low-stable-rank} follows the (by now) standard geometric framework pioneered by Rudelson and Vershynin \cite{RV08}. 
The key deviation in our argument is in showing that a unit normal vector to the random hyperplane spanned by the (say) first $n-1$ rows of the matrix is arithmetically very unstructured (in the sense of having exponentially large Least Common Denominator (LCD)). The union bound argument in \cite{RV08} or its refinement in \cite{livshyts2019smallest} is inadequate here since the matrix $A$ may be arbitrarily large. Our innovation (\cref{lem:metric-entropy-LCD}) is to execute a union bound argument by covering level sets of the LCD by certain oblique convex sets adapted to the geometry of $A$, which exploits the fact that it suffices to cover only those vectors $v$ which are ``approximately orthogonal'' to the large singular vectors of $A$.

Throughout this section, we will fix $K$ and $\eta$ as in the statement of \cref{thm:low-stable-rank}, and let $B$ and $N$ denote the $(n-1)\times n$ matrices consisting of the first $n-1$ rows of $A$ and $M$ respectively. Note that the $B$ has at least $\eta n/2$ singular values smaller than $K\sqrt{n}$, as is readily seen by applying the Cauchy interlacing theorem to the matrix $AA^{T}$ and its minor $BB^{T}$. 
Throughout, we will let $V \subseteq \mb{R}^{n-1}$ denote the span of the bottom $\eta n/2$ left-singular vectors of $B$ (note that the singular values corresponding to these singular vectors are all at most $K\sqrt{n}$). Moreover, $P_{V}$ will denote the orthogonal projection operator from $\mb{R}^{n-1}$ onto $V$. 

We will use the decomposition of the unit sphere into compressible and incompressible vectors, formalized by Rudelson and Vershynin \cite{RV08}. 
\begin{definition}[Compressible and Incompressible vectors, {\cite[Definition~3.2]{RV08}}]Fix $\delta, \rho \in (0,1)$. A vector $x \in \mb{R}^{N}$ is said to be \emph{sparse} if $|\on{supp}(x)| \le \delta n$. A vector $x \in \mb{S}^{N-1}$ \emph{compressible} if it is within Euclidean distance $\rho$ of a sparse vector, and \emph{incompressible} otherwise. We denote the set of compressible and incompressible vectors by $\on{Comp}_{N}(\delta, \rho)$ and $\on{Incomp}_{N}(\delta, \rho)$ respectively, dropping the subscript $N$ when the underlying dimension is clear from context.
\end{definition}

We will also need the notion of the L\'evy concentration function. 
\begin{definition}[L\'evy concentration function]For a random variable $X$ and a real number $r \ge 0$, we define the L\'evy concentration function of $X$ at radius $r$ by 
\[\mc{L}(X, r) = \sup_{y\in \mb{R}}\mb{P}[|X-y|\le r].\]

\end{definition}

First, we show that any fixed vector $v \in \mb{S}^{n-1}$ has exponentially small probability of being in the kernel of the first $n-1$ rows of $(A + M)$. In fact, we will show something stronger: for every fixed vector $v \in \mb{S}^{n-1}$, its image under $P_V(B+N)$ has norm $\Omega(\sqrt{n})$ except with exponentially small probability; this strengthening will be crucial for the union bound argument in \cref{lem:compressible-vectors}.

The next lemma follows immediately from  \cite[Corollary~6]{tikhomirov2016smallest}, which is a direct consequence of the main result in \cite{RV15}. For sub-Gaussian random variables, it follows easily from the Hanson-Wright inequality (straightforwardly modifying the proofs of \cite[Corollary~2.4]{RV13} and \cite[Corollary~3.1]{RV13}); we include the deduction for  the reader's convenience.


\begin{lemma}[Invertibility on a single vector]\label{lem:fixed-vector}
There exists $c_{\ref{lem:fixed-vector}} > 0$ depending only on $\eta$ and the sub-Gaussian norm of $\xi$ such that for any $v \in \mb{S}^{n-1}$, 
\[\mb{P}[\snorm{P_V(B+N)v}_2\le c_{\ref{lem:fixed-vector}}\sqrt{n}]\le 2e^{-c_{\ref{lem:fixed-vector}}n}.\]
\end{lemma}
\begin{proof}
Let $N$ and $N'$ be independent copies of $N$. Then
\begin{align*}
\mb{P}[\snorm{P_V(B+N)v}_2\le\sqrt{\eta n}/4]^2 &=   \mb{P}[\snorm{P_V(B+N)v}_2\le\sqrt{\eta n}/4 \cap \snorm{P_V(B+N')v}_2\le \sqrt{\eta n}/4]\\
&\le\mb{P}[\snorm{P_V(N-N')v}_2\le\sqrt{\eta n}/2].
\end{align*}
Note that the coordinates of $(N-N')$ are i.i.d sub-Gaussian random variables with variance $2$, and sub-Gaussian norm depending only on that of $\xi$. 
In particular, 
\[\mb{E}[\snorm{P_V(N-N')v}_2^2] = 2\on{Tr}(P_V^TP_V)\ge\eta n.\]
The result now follows from the Hanson-Wright inequality applied to the matrix $P_V$ and vector $(N-N')v$ (see \cite[Theorem~2.1]{RV13}), which shows that
\[\mb{P}[|\snorm{P_V(N-N')v}_2 - 2\on{Tr}(P_V^TP_V)|>t]\le 2\exp(c_\xi t^2),\]
where $c_\xi$ depends only on the sub-Gaussian norm of $\xi$.
\end{proof}

Given this, we can quickly derive invertibility on compressible vectors. 
\begin{lemma}[Invertibility on compressible vectors]\label{lem:compressible-vectors}
There exist $\delta_{\ref{lem:compressible-vectors}},\rho_{\ref{lem:compressible-vectors}},c_{\ref{lem:compressible-vectors}} > 0$ depending only on $\eta$, $K$, and the sub-Gaussian norm of $\xi$ for which
\[\mb{P}[\exists v\in\on{Comp}_{\delta_{\ref{lem:compressible-vectors}},\rho_{\ref{lem:compressible-vectors}}}: \snorm{(B+N)v}_2\le c_{\ref{lem:compressible-vectors}}\sqrt{n}]\le 2e^{-c_{\ref{lem:compressible-vectors}}n}.\]
\end{lemma}
\begin{proof}
The key point is that, by definition of $V$,  $\snorm{P_VB}\le K\sqrt{n}$. 
Since $\xi$ has bounded sub-Gaussian norm, it follows (cf. \cite[Lemma~2.4]{RV08}) that there exists $K'$, depending only on $K$ and the sub-Gaussian norm of $\xi$, for which the event
$\mc{E}_{K'} = \{\snorm{P_V(B+N)}\le K'\sqrt{n}\}$ occurs with probability at most $2e^{-cn}$,
for some $c > 0$ depending only on the sub-Gaussian norm of $\xi$. 

Let $\delta, \rho, c' > 0$, and consider an $\epsilon$-net $\mc{N}$ of $\on{Comp}_{\delta,\rho}$. Then, we have
\begin{align*}
\mb{P}[\exists v\in\on{Comp}_{\delta,\rho}&: \snorm{(B+N)v}_2\le c'\sqrt{n}]\\
&\le\mb{P}[\mc{E}_{K'}]+\mb{P}[\exists v\in\on{Comp}_{\delta,\rho}: \snorm{P_V(B+N)v}_2\le c'\sqrt{n}\wedge\mc{E}_{K'}^c]\\
&\le 2e^{-cn}+\mb{P}[\exists v\in\mc{N}: \snorm{P_V(B+N)v}_2\le (c'+K'\epsilon)\sqrt{n}].
\end{align*}
Therefore, by using \cref{lem:fixed-vector}, we can conclude, provided we first choose $c', \epsilon$ small enough so that
$c'+K'\epsilon < c_{\ref{lem:fixed-vector}
}$ (this is clearly possible), and then choose $\delta,\rho$ small enough so that there is an $\epsilon$-net $\mc{N}$ of $\on{Comp}_{\delta,\rho}$ of size at most $\exp(c_{\ref{lem:fixed-vector}}n/2)$ (this is possible by a standard volumetric bound on the size of $\epsilon$-nets of $\on{Comp}_{\delta, \rho}$, cf. \cite[Lemma~4.3]{Rud14}). We omit the  standard details.
\end{proof}

For invertibility on incompressible vectors, we will use the following reduction due to Rudelson and Vershynin.
\begin{lemma}[Invertibility via distance, {\cite[Lemma~3.5]{RV08}}]
\label{lem:invertibility-distance}
Let $A$ be a random $n\times n$ matrix. Let $X_i$ denote the column vectors of $A$, and let $H_k = \on{span}(X_{-k})$ denote the span of all columns except for the $k^{th}$ column. Then for every $\delta,\rho\in(0,1)$ and $\epsilon\ge 0$ we have that 
\[\mb{P}\left[\inf_{x\in\on{Incomp}(\delta,\rho)}\snorm{Ax}_2\le \epsilon\rho n^{-1/2}\right]\le \frac{1}{\delta n}\sum_{k=1}^{n}\mb{P}[\on{dist}(X_k,H_k)\le \epsilon].
\]
\end{lemma}
Given this, our goal is to find a uniform (in $k \in [n]$) upper bound on $\mb{P}[\on{dist}(X_k, H_k) \le \epsilon]$.

We will need the crucial notion of the Least Common Denominator due to Rudelson and Vershynin.



\begin{definition}[Least common denominator, cf. \cite{Rud14}]\label{def:LCD}
For a vector $v\in\mb{R}^N$, $\gamma \in (0,1)$, and $\alpha > 0$, we define
\[\on{LCD}_{\alpha,\gamma}(v) = \inf\{\theta > 0: \on{dist}(\theta v,\mb{Z}^{N}) < \min(\gamma\snorm{\theta v},\alpha)\}.\]
\end{definition}

We collect some useful properties of the LCD. 

\begin{lemma}[Lower bound on LCD, cf. {\cite[Lemma~6.1]{Rud14}}]
\label{lem:lower-bound-LCD}
Fix $\delta,\rho>0$. There exist $\gamma,\lambda>0$ (depending only on $\delta,\rho$) such that for any $v\in\on{Incomp}_N(\delta,\rho)$ and for all $\alpha > 0$, we have $\on{LCD}_{\alpha, \gamma}(v)\ge \lambda \sqrt{n}$.
\end{lemma}

\begin{lemma}[Anti-concentration via LCD]\label{lem:levy-concentration-LCD}
For any $\gamma\in(0,1)$, there exist $c_{\ref{lem:levy-concentration-LCD}}(\gamma), C_{\ref{lem:levy-concentration-LCD}}(\gamma) > 0$ depending only on $\gamma$ and the sub-Gaussian norm of $\xi$ for which the following holds. Let $v\in\mb{S}^{N-1}$. Then, for every $\alpha > 0$ and $\epsilon\ge 0$,
\[\mc{L}\left(\sum_{i=1}^{n}v_i\xi_i,\epsilon\right)\le C_{\ref{lem:levy-concentration-LCD}}(\gamma)\epsilon + \frac{C_{\ref{lem:levy-concentration-LCD}}(\gamma)}{\on{LCD}_{\alpha,\gamma}(v)} + C_{\ref{lem:levy-concentration-LCD}}(\gamma)e^{-c_{\ref{lem:levy-concentration-LCD}}(\gamma)\alpha^2},\]
where $\xi_1,\dots,\xi_{n}$ denote i.i.d copies of $\xi$.
\end{lemma}

Given \cref{lem:levy-concentration-LCD}, a good upper bound on $\mb{P}[\on{dist}(X_k, H_k) \le \epsilon]$ may be obtained by showing that it is exponentially unlikely that any unit normal to $H_k$ has subexponential LCD.  The proof of Rudelson and Vershynin in the centered sub-Gaussian case accomplishes this by decomposing the set of incompressible vectors into level sets of the LCD (\cref{def:level-set-LCD}), and then finding, for each level set, a net (at appropriate scale) of small enough size to survive a union bound argument. We note that this step requires considerable care, and there is not much room available in the entropy-energy trade-off.  
\begin{definition}[Level sets of LCD]\label{def:level-set-LCD}
Let $D, \gamma, \mu > 0$. We define
\[S_D = \{x\in\mb{S}^{n-1}: D\le\on{LCD}_{\mu\sqrt{n},\gamma}(x)\le 2D\}.\]
\end{definition}

In our setting, such a net argument for $S_D$ has no hope of working, since the operator norm of $A+M$ can be arbitrarily large. The key idea to overcome this is to exploit the fact that for a sufficiently large constant $K'$, the following holds except with exponentially small probability:  $\snorm{Bv}_{2} > 2K'\sqrt{n} $ implies that $\snorm{(B + N)v}_2 > K'\sqrt{n}$. This motivates the following definition. 
\begin{definition}[Restricted level sets of the LCD]
For $K' > 0$, let
\[G_{K'} = \{\snorm{x}_2\le 1\cap\snorm{Bx}_2\le 2K'\sqrt{n}\},\]
and let 
$$S_D'(K') = S_D\cap G_{K'}.$$
\end{definition}
In order to find an efficient net for $S_D'(K')$, we will use the following two easy geometric observations repeatedly. Note that the remainder of the proofs, parallelepipeds will always be ``right-parallelepipeds'' or more colloquially ``boxes''.
\begin{lemma}\label{lem:cube-covering}
There exists an absolute constant $C_{\ref{lem:cube-covering}} \ge 1$ such that any ellipsoid with semiaxes $\ell_i$ can be covered by at most $C_{\ref{lem:cube-covering}}^{n}$  parallelepipeds with widths $\ell_i/\sqrt{n}$ in the direction of the semiaxes.
\end{lemma}
\begin{proof}
This follows from a standard volumetric argument. 
\end{proof}
\begin{lemma}\label{lem:set-reduction}
The sets $S_1 = \{\snorm{x}_2\le 1\cap\snorm{Jx}_2\le 1\}$ and $S_2 = \{\snorm{x}_2^2+\snorm{Jx}_2^2\le 1\}$ satisfy $S_2\subseteq S_1\subseteq\sqrt{2}S_2$.
\end{lemma}
\begin{proof}
This is immediate. 
\end{proof}
We will also need the following elementary lattice counting fact.
\begin{lemma}\label{lem:counting-points}
The number of integer lattice points in a right-parallelepiped with dimensions $\ell_i\ge 1$ is at most $C_{\ref{lem:counting-points}}^n\prod_{i=1}^n\ell_i$, where $C_{\ref{lem:counting-points}} \ge 1$ is an absolute constant. 
\end{lemma}
\begin{remark}
Note the parallelepiped does not need to be axis-aligned or centered.
\end{remark}
\begin{proof}
By expanding the dimensions to $\lceil\ell_i\rceil\le 2\ell_i$, and by decomposing the region into unit cubes whose axes are aligned with those of the parallelepiped, it suffices to show that there is an absolute constant $C \ge 1$ such that any rotated unit cube has at most $C^{n}$ integer lattice points in it. Since any rotated unit cube is contained in a ball of radius $\sqrt{n}$, it suffices to show that any ball (not necessarily centered) of radius $\sqrt{n}$ contains at most $C^{n}$ integer lattice points, for an absolute constant $C \ge 1$, which follows from a standard well-known volumetric argument.  
\end{proof}

We are now ready to state and prove our key new ingredient. 

\begin{lemma}[Nets of level sets of LCD]\label{lem:metric-entropy-LCD}
Fix $\lambda, K' > 0$. There is $C_{\ref{lem:metric-entropy-LCD}} = C_{\ref{lem:metric-entropy-LCD}}(\lambda, K,K')$ for which the following holds. For all sufficiently small $\mu > 0$ (depending on $\lambda$) and for all $D \ge \lambda \sqrt{n}$,  
there exists a net $\mc{N}$ of cardinality at most $\mu^{-(1-\eta/2)n}D^2(C_{\ref{lem:metric-entropy-LCD}}D/\sqrt{n})^n$ with the following property: on the event $\snorm{N}\le K'\sqrt{n}$, for any $v\in S_D'(K')$, there is $w\in\mc{N}$ satisfying
\[\snorm{(B+N)(v-w)}_2\le\frac{\mu n}{D}.\]
\end{lemma}
\begin{proof}
For lightness of notation, we will denote $S'_D(K')$ by $S'_D$. We denote the singular values of $B$, in decreasing order, by $\sigma_1(B),\dots, \sigma_n(B)$, and the corresponding unit right-singular vectors by $b_1,\ldots,b_n$. In particular, $b_1,\dots, b_n$ form an orthonormal basis of $\mb{R}^{n}$. Note that since $B$ is an $(n-1)\times n$ matrix, we must have $\sigma_n(B) = 0$, and so $b_n$ is in the kernel of $B$. The proof will make use of a few different geometric objects; we collect their definitions here for ease of reference:
\begin{itemize}
    \item $\mc{C}$ denotes the $n$-dimensional cube of width $\mu/D$ centered at $0$, whose axes are aligned with $b_1,\dots, b_n$.  
    \item $\mc{G}$ denotes the region given by $\{x \in \mb{B}_2^{n}: \snorm{Bx}_2\le 2K'\sqrt{n}\}$.
    \item $\mc{Q}$ denotes the ellipsoid given by $\{4K'^2n\snorm{x}_2^2+\snorm{Bx}_2^2\le 8K'^2n\}$.
    \item $\mc{C}'$ denotes the right-parallelepiped centered at $0$ whose axes are aligned with $b_1,\dots, b_n$, and whose width along $b_i$ is
    \[\sqrt{\frac{8K'^{2}}{\sigma_i(B)^{2} + 4K'^{2}n}};\]
    note that $\mc{C}' \subseteq \mc{Q}$.
    \item $\mc{C}''$ denotes the right-parallelepiped centered at $0$ whose axes are aligned with $b_1,\dots, b_n$, and whose width along $b_i$ is
    \[\sqrt{\frac{8K'^{2}}{\sigma_i(B)^{2} + 4K'^{2}n}} + \frac{1}{D};\]
    note that $\mc{C}' + \mu^{-1}\mc{C} \subseteq \mc{C}''$.
    \item $\mc{C}'''$ denotes the right-parallelepiped centered at $0$ whose axes are aligned with $b_1,\dots, b_n$, and whose width along $b_i$ is
    \[\frac{1}{\min(D, \max(\sigma_i(B), \sqrt{n}))}\]
\end{itemize}

\textbf{Step 1: }From the definition of LCD, it follows that $S_D$ admits a $(2\mu\sqrt{n}/D)$-net in Euclidean norm, formed by the points
\[\mc{P} = \bigg\{\frac{p}{\snorm{p}_2}: p\in(\mb{Z}^n \setminus \{0\})\cap\mb{B}_2^n(0,3D)\bigg\}= \bigg\{\frac{p}{\snorm{p}_2}: p\in(\mb{Z}^n \setminus \{0\}) \cap(\mb{B}_2^n(0,3D)\setminus\mb{B}_2^n(0,3D/2))\bigg\}.\]
For a concrete reference for the first equality, see \cite[Lemma~7.2]{Rud14}. For the second equality, simply note that for any $p \in (\mb{Z}^{n} \setminus \{0\}) \cap \mb{B}_2^{n}(0, 3D/2)$, there exists some $\ell_{p} \in \mb{Z}\setminus \{0\}$ such that $\ell_{p}\cdot p \in (\mb{Z}^{n} \setminus \{0\}) \cap (\mb{B}_{2}^{n}(0, 3D) \setminus \mb{B}_2^{n}(0, 3D/2))$. Since  $\ell_p \cdot p/\snorm{\ell_p \cdot p}_2 = p/\snorm{p}_2$, both $p$ and $\ell_p \cdot p$ map to the same point in $\mc{P}$, so that we may ignore the contribution of $p$ without any loss. 

Hence, we see that $S_D \subseteq \mc{P} + \mb{B}_{2}^{n}(0, 2\mu\sqrt{n}/D)$, so that 
\[S_D' \subseteq (\mc{P} + \mb{B}_{2}^{n}(0, 2\mu \sqrt{n}/D)) \cap \mc{G}.\]
In the remainder of the proof, we will cover the region on the right hand side by at most \[\mu^{-(1-\eta/2)n}D^{2}(CD/\sqrt{n})^{n}\] translates of 
\[\frac{\mu \sqrt{n}}{4K'D}\mc{Q};\]
this clearly suffices since on the event $\snorm{N} \le K'\sqrt{n}$, for any $v, w \in \mu\sqrt{n}\mc{Q}/4K'D$, we have $\snorm{(B+N)(v-w)}_2 \le \mu n/D$. 

\textbf{Step 2: }Since for any $p \in (\mb{Z}^{n}\setminus \{0\}) \cap \mb{B}_{2}^{n}(0, 3D)$, $\snorm{p}_2^{2} \in \mb{Z} \cap [1, 9D^{2}]$, it follows that by paying an overall multiplicative factor of $9D^{2}$ in the size of the final net, it suffices to fix $T \in [3D/2, 3D]$ and bound (uniformly in $T$) the number of translates of $\mu \sqrt{n} \mc{Q}/4K'D$ needed to cover the region
\[(\mc{P}_T + \mb{B}_{2}^{n}(0, 2\mu \sqrt{n}/D)) \cap \mc{G},\]
where
\[\mc{P}_T = \left(\frac{1}{T}\mb{Z}^n\right)\cap\mb{B}_2^n.\]
Moreover, since $\mb{B}_{2}^{n}(0, 2\mu \sqrt{n}/D)$ can be covered by $(2C_{\ref{lem:cube-covering}})^{n}$ translates of $\mc{C}$ (\cref{lem:cube-covering}), it suffices after paying a multiplicative factor of $(2C_{\ref{lem:cube-covering}})^{n}$ to bound (uniformly in $T$ and $y \in \mb{R}^{n}$) the number of translates of $\mu \sqrt{n}\mc{Q}/4K'D$ needed to cover the region
\[(y + \mc{P}_T + \mc{C}) \cap \mc{G}.\]
Moreover, by \cref{lem:set-reduction}, it suffices instead to consider the larger region
\[(y + \mc{P}_T + \mc{C}) \cap \mc{Q}.\]

Note that covering $\mc{Q}$ by translates of $\mu \sqrt{n} \mc{Q}/4K'D$ requires 
$\mu^{-n}(CD/\sqrt{n})^n$ translates, which is bigger by a factor of $\mu^{\eta n/2}$ compared to our desired conclusion; therefore, we must carefully exploit the first term in the intersection. 

\textbf{Step 3: }By noting that 
the matrix $(4K'^2nI + B^TB)/(8K'^2n)$ has unit right-eigenvectors $b_1,\dots, b_n$ with corresponding eigenvalues $(4K'^{2} n + \sigma_i(B)^{2})/8K'^{2}n$, it follows from \cref{lem:cube-covering} that $\mc{Q}$ can be covered by $C_{\ref{lem:cube-covering}}^{n}$ translates of $\mc{C}'$. 
Therefore, up to an overall multiplicative factor of $C_{\ref{lem:cube-covering}}^{n}$, it suffices to bound (uniformly in $T$, and $y_1, y_2 \in \mb{R}^{n}$) the number of translates of $\mu \sqrt{n} \mc{Q}/4K'D$ needed to cover the region 
\[(y_1 + \mc{P}_T + \mc{C}) \cap (y_2 + \mc{C}').\]
Moreover, since $\mc{C}'\subseteq \mc{Q}$, it suffices to cover by translates of $\mu \sqrt{n}\mc{C}'/4K'D$.
Below, we will need the following observation. Let 
\[k^* = \max\{i \in [n]: \sigma_i(B) \ge KD/\lambda\}.\]
Then, since $D \ge \lambda \sqrt{n}$ by assumption, and recalling the observation that $B$ has $K$-rank at most $(1-\eta/2)n$, it follows that 
$k^*\le (1-\eta/2)n$.

\textbf{Step 4: }We begin by bounding (from above) the number of points $w\in y_1+\mc{P}_T$ for which \[(w + \mc{C}) \cap (y_2 + \mc{C}') \neq \emptyset.\]  
Note that any such $w$ is of the form $y_1 + (z/T)$, where $z \in \mb{Z}^{n}$, and for the intersection above to be nonempty, we must have
\[(z/T) \in (y_2 - y_1) + \mc{C}' - \mc{C} \subseteq (y_2 - y_1) + \mc{C}''.\]

We claim that $\mc{C}'' \subseteq C\cdot \mc{C'''}$, where $C$ is a constant depending only on $K, K', \lambda$. Indeed, using the definition of $\mc{C}'', k^*$, and the assumption $D \ge \lambda \sqrt{n}$, we see that the width of $\mc{C}''$ in the directions $b_i$ for $i \le k^*$ is at most $C''/D$, whereas its width in the directions $b_i$ for $i > k^*$ is at most $C''/\max(\sigma_i(B), \sqrt{n})$, where $C''$ is a constant depending on $K, K', \lambda$.  

Hence, it suffices to bound the number of $z \in \mb{Z}^{n}$ such that $z \in T\cdot (y_2 - y_1) + T\cdot C\cdot \mc{C'''}$. Moreover, since $T \le 3D$, it suffices to bound (uniformly in $y \in \mb{R}^{n}$) the number of $z \in \mb{Z}^{n}$ such that $z \in y + 3CD\cdot \mc{C'''}$. Since for all sufficiently large $C$ (depending only on $K,K', \lambda$),
\[\frac{3CD}{\min(D, \max(\sqrt{n}, \sigma_i(B)))} \ge 1,\]
it follows from 
\cref{lem:counting-points} that the number of such $z \in \mb{Z}^{n}$ is at most 
\[N_1 := \frac{(CD)^n}{\prod_{i=1}^n(\min(D, \max(\sqrt{n}, \sigma_i(B))))},\]
where $C$ is a constant depending only on $K, K', \lambda$.

\textbf{Step 5: }Let us now fix $w_0 \in y_1 + \mc{P_T}$ and bound (uniformly in $y_1, y_2 \in \mb{R}^{n}$) the number of translates of $\mu \sqrt{n} \mc{C}'/4K'D$ needed to cover the region 
\[(w_0 + \mc{C}) \cap (y_2 + \mc{C}').\]
For this, note that for any $y_2 \in \mb{R}^{n}$, $(w_0 + \mc{C}) \cap (y_2 + \mc{C'})$ is the intersection of two right-parallelepipeds along the same axes $b_1,\dots,b_n$. In particular, in each direction $b_i$, the width of $(w_0 + \mc{C}) \cap (y_2 + \mc{C'})$ is bounded above by the minimum of the widths of $\mc{C}$ and $\mc{C}'$ along $b_i$. We will use the bound for the width coming from $\mc{C}'$ for the directions $b_i$, $i \le k^*$, and the bound for the width coming from $\mc{C}$ for the remaining directions. 

Thus, the width of $\mu \sqrt{n}\mc{C'}/4K'D$ is smaller than the width of $(w_0 + \mc{C}) \cap (y_2 + \mc{C}')$ by a factor of at most $\mu \sqrt{n}/4K'D$ in the directions $b_i$ for $i \le k^*$. Moreover, for $i > k^*$, the width of $(w_0 + \mc{C}) \cap (y_2 + \mc{C}')$ in direction $b_i$ is at most $\mu/D$, whereas the width of $\mu \sqrt{n} \mc{C'}/4K'D$ is at least $\mu \sqrt{n}/C'D\max(\sigma_i(B), \sqrt{n})$, where $C'$ is a constant depending only on $K, K', \lambda$; hence, in these directions, the width of $\mu \sqrt{n} \mc{C}'/4K'D$ is smaller by a factor of at most $\sqrt{n}/C'\max(\sigma_i(B), \sqrt{n})$. 

Therefore, we see that the number of translates of $\mu \sqrt{n}\mc{C}'/4K'D$ needed to cover $(w_0 + \mc{C}) \cap (y_2 + \mc{C}')$ is at most
\[N_2 := C^{n}\prod_{i=1}^{k^*}\left(\frac{D}{\mu \sqrt{n}}\right)\prod_{i = k^* + 1}^{n}\left(\frac{\max(\sigma_i(B), \sqrt{n})}{\sqrt{n}}\right),\]
where $C$ is a constant depending only on $K, K', \lambda$. 



\textbf{Step 6: }Noting that 
\[\prod_{i=1}^{n}\min(D, \max(\sqrt{n}, \sigma_i(B))) \ge (C')^{-n} D^{k^*}\prod_{k^* + 1}^{n}\max(\sigma_i(B), \sqrt{n}),\]
where $C'$ is a constant depending only on $K, K', \lambda$,
we see that
\[N_1 \cdot N_2 \le \mu^{-k^*}\left(\frac{C D}{\sqrt{n}}\right)^{n} \le \mu^{-(1-\eta/2)n}\left(\frac{C D}{\sqrt{n}}\right)^{n}.\]
Finally, recalling that $N_1 \cdot N_2$ is less than the size of actual number of translates by at most $C^{n}D^{2}$, where $C$ depends on $K, K', \lambda$, gives the desired conclusion. 
\end{proof}

We are now ready to prove \cref{thm:low-stable-rank}; the proof is essentially identical to the proof in \cite{RV08} except for one twist. 
\begin{proof}[Proof of \cref{thm:low-stable-rank}]
By \cref{lem:compressible-vectors} applied to $A+M$ and $(A+M)^{T}$, it follows that there exist $\delta_{c_{\ref{lem:compressible-vectors}}}, \rho_{c_{\ref{lem:compressible-vectors}}}, c_{\ref{lem:compressible-vectors}} > 0$ depending only on $\eta$, $K$, and the sub-Gaussian norm of $\xi$ such that except with probability $4e^{-c_{\ref{lem:compressible-vectors}}n}$, any left or right compressible vector has image with norm at least $c_{\ref{lem:compressible-vectors}}\sqrt{n}$.  
Then, by \cref{lem:invertibility-distance}, it suffices to provide a uniform (in $k \in [n]$) bound on $\mb{P}(\on{dist}(X_k, H_k) \le \epsilon)$, where $X_k$ denotes the $k$th row of $A+M$, and $H_k$ denotes the span of all the other rows of $A+M$. Henceforth, we will take $k=n$, noting that the argument for other values of $k$ is exactly the same. 

Let $v(\cdot)$ be a function mapping $(n-1)\times n$ matrices to an arbitrary unit vector in their right kernel. 
As before, let $B$ denote the $(n-1)\times n$ matrix formed by the top $n-1$ rows of $A$, and let $N$ denote the $(n-1)\times n$ matrix formed by the top $n-1$ rows of $M$. By \cref{lem:compressible-vectors}, it follows that, except with probability at most $2e^{-c_{\ref{lem:compressible-vectors}}n}$,  $v(B+N)$ must lie in $\on{Incomp}_{\delta_{\ref{lem:compressible-vectors}}, \rho_{\ref{lem:compressible-vectors}}}$. Therefore, by \cref{lem:lower-bound-LCD}, there exist some $\gamma, \lambda > 0$ (depending only on $\delta_{\ref{lem:compressible-vectors}}, \rho_{\ref{lem:compressible-vectors}}$) such that $\on{LCD}_{\alpha, \gamma}(v(B+N)) \ge \lambda \sqrt{n}$ for any $\alpha > 0$. We write $\alpha = \mu \sqrt{n}$, where $\mu > 0$ will be chosen later; it is important to note that none of the previously chosen parameters depend on $\mu$.  

Let $X_n$ denote the last row of $A + M$. Since $\on{dist}(X_n, H_n) \ge |\sang{v(B+N), X_n}|$, it suffices to bound
$$\mb{P}[|\sang{v(B+N), X_n}|\le \epsilon]. $$

First, we consider, for dyadically chosen $D\in[\lambda\sqrt{n},2^{\chi n}]$ (where $\chi > 0$ will be chosen at the end), the probability
\[\mb{P}[|\sang{v(B+N),X_n}|\le\epsilon\wedge v(B+N)\in S_D].\]
The key observation here is that this probability is at most
\[\mb{P}[|\sang{v(B+N),X_n}|\le\epsilon\wedge v(B+N)\in S'_D(K')] + \mb{P}[\snorm{N} \ge K'\sqrt{n}],\]
since for any $x \in S'_D(K')\setminus S_D$, on the event $\snorm{N}\le K'\sqrt{n}$, we have $\snorm{(B+N)x} \ge 2K'\sqrt{n} - K'\sqrt{n} \neq 0$.
Note that there exists $K' > 0$ depending only on the sub-Gaussian norm of $\xi$ for which the second term is exponentially small (cf. \cite[Lemma~2.4]{RV08}). We fix such a $K'$. 

Now, we bound the first term as follows. Let $\mc{N}$ denote the net for $S_D'(K')$ coming from \cref{lem:metric-entropy-LCD}. Then,
\begin{align*}
\mb{P}[|\sang{v(B+N),X_n}|&\le\epsilon\wedge v(B+N)\in S'_D(K')]\\
&\le \mb{P}[v(B+N) \in S_D'(K')]\\
&\le\sum_{w\in\mc{N}}\mb{P}[|(B+N)w|\le\mu n/D]\\
&\le\mu^{-(1-\eta/2)n}D^2\bigg(\frac{C_{\ref{lem:metric-entropy-LCD}}(\lambda, K, K')D}{\sqrt{n}}\bigg)^n\times\bigg(\frac{C\cdot C_{\ref{lem:levy-concentration-LCD}}(\gamma)\mu\sqrt{n}}{D}\bigg)^{n-1}, 
\end{align*}
which is exponentially small (here, $C$ is an absolute constant), as long as $\mu > 0$ is chosen to be sufficiently small, and then $\chi$ is chosen to be sufficiently small. In the last line, we have used \cref{lem:metric-entropy-LCD}, \cref{lem:levy-concentration-LCD} and a standard tensorization argument (cf. \cite[Lemma~2.2]{RV08}).

Finally, we note that by \cref{lem:invertibility-distance}, we have
\[\mb{P}[\sang{v(B+N),A_n+M_n}\le\epsilon\wedge\on{LCD}_{\mu\sqrt{n},\gamma}(v(B+N))\ge 2^{\chi n}]\le C_{\ref{lem:levy-concentration-LCD}}(\gamma)(\epsilon + 2^{-\chi n} + 2^{-c\mu^2n}),\]
which completes the proof. 
\end{proof}

\section{Proof of \texorpdfstring{\cref{thm:counterexample}}{Theorem 1.3}}\label{sec:counterexample}
Let $A,R_n$ be defined as in the statement of \cref{thm:counterexample}. For convenience, we will study the least singular value of $A-R_n$, which has the same distribution as $A+R_n$. Let $K\ge 1$ be sufficiently large so that the event $\mc{E}_K = \{\snorm{R_n}\le K\sqrt{n}\}$ satisfies
\[\mb{P}[\mc{E}_K]\ge 1-2^{-n}.\]
Finally, let $L\ge 2K\sqrt{n}$.

We first reduce the study of the smallest singular value to the study of anti-concentration of a very structured vector.
\begin{lemma}\label{lem:reduction}
Let $u,w\in\mb{R}^{n-1}$ have i.i.d coordinates distributed as lazy Rademacher random variables. Then, we have
\[\frac{1}{2}\mb{P}\bigg[\bigg|\sum_{i\ge 0}w^T(R_{n-1}/L)^iu\bigg|\le L\epsilon\cap\snorm{R_{n-1}}\le K\sqrt{n}\bigg]\le\mb{P}[s_n(A-R_n)\le\epsilon]+\mb{P}[\mc{E}_K^c].\]
\end{lemma}
\begin{proof}
Since we have the term $\mb{P}[\mc{E}_K^{c}]$ on the right hand side, we may henceforth restrict ourselves to the event $\snorm{R_{n-1}} \le \snorm{R_{n}} \le K \sqrt{n}$. On this event, note that for any $L \ge 2K\sqrt{n}$, the $(n-1)\times (n-1)$ random matrix $LI_{n-1} - R_{n-1}$ is invertible, and in fact, its inverse is expressible as a Neumann series, i.e.
$$(LI_{n-1} - R_{n-1})^{-1} = \sum_{i\ge 0}R_{n-1}^{i}/L^{i+1}.$$
The key is to decompose $R_n$ as 
\[
R_{n}=\left(\begin{array}{cc}
R_{n-1} & u\\
w^{T} & r_{nn}
\end{array}\right);
\]
note that the four block matrices appearing in the decomposition are independent of each other. Therefore, for
\[
v=\left(\begin{array}{c}
(LI_{n-1}-R_{n-1})^{-1}u\\
1
\end{array}\right),
\]
we have trivially (from the last coordinate) that $\snorm{v}_2 \ge 1$, and hence,
\begin{align*}
    s_n(A - R_n) 
    &\le \snorm{(A-R_n)v}\\
    &\le |w^{T}(L-R_{n-1})^{-1}u + r_{nn}|\\
    &= |w^{T}(L-R_{n-1})^{-1}u|,
\end{align*}
where the final equality holds on the event $r_{nn} = 0$. 

Finally, writing $(L-R_{n-1})^{-1}u = L^{-1}\sum_{i\ge 0}(R_{n-1}/L)^{i}u$, we have



\begin{align*}
\frac{1}{2}\mb{P}\bigg[\bigg|\sum_{i\ge 0}w^T(R_{n-1}/L)^iu\bigg|&\le L\epsilon\cap\{\snorm{R_{n-1}}\le K\sqrt{n}\}\bigg]\\
&= \mb{P}\bigg[\bigg|\sum_{i\ge 0}w^T(R_{n-1}/L)^iu\bigg|\le L\epsilon\cap\{\snorm{R_{n-1}}\le K\sqrt{n}\}\cap \{r_{nn} = 0\}\bigg]\\
&\le\mb{P}\bigg[\bigg|\sum_{i\ge 0}w^T(R_{n-1}/L)^iu\bigg|\le L\epsilon\cap\mc{E}_K\cap \{r_{nn} = 0\}\bigg]+\mb{P}[\mc{E}_K^c]\\
&\le\mb{P}[s_n(A-R_n)\le\epsilon]+\mb{P}[\mc{E}_K^c]. \qedhere
\end{align*}
\end{proof}
We are now in position to prove \cref{thm:counterexample}.
\begin{proof}[Proof of \cref{thm:counterexample}]
Fix an integer $t\ge 1$, and let $C > 0$ be an absolute constant to be chosen later. Let 
$$\mc{G}_K = \{Q \in \mb{R}^{(n-1)\times (n-1)} : \snorm{Q} \le K\sqrt{n}\}.$$

\textbf{Step 1: }For independent $R_{n-1}, u, w$, consider the event
$$\mc{E} =  \bigcap_{i=0}^{t-2}\{w^TR_{n-1}^iu = 0\} \cap \left\{\bigg|w^T\bigg(\sum_{i\ge t-1}(R_{n-1}/L)^i\bigg)u\bigg|\le\frac{CK^tn^{t/2}}{L^{t-1}}\right\}.$$
Then, by \cref{lem:reduction} and the bound on $\mb{P}[\mc{E}_K]$, we have
\begin{align*}
\mb{P}\bigg[s_n(A-R_n)&\le C\bigg(\frac{K\sqrt{n}}{L}\bigg)^t\bigg]\\
&\ge\frac{1}{2}\mb{P}\bigg[\bigg|\sum_{i\ge 0}w^T(R_{n-1}/L)^iu\bigg|\le\frac{CK^{t-1}n^{t/2}}{L^{t-1}}\cap\{\snorm{R_{n-1}}\le K\sqrt{n}\}\bigg]-2^{-n}\\
&\ge\frac{1}{2}\mb{P}[\mc{E} \cap \{R_{n-1} \in \mc{G}_K\}]-2^{-n}.
\end{align*}

\textbf{Step 2: }We write $u = u_1-u_2$, where $u_1,u_2\in\mb{R}^n$ are independent random vectors with i.i.d coordinates distributed as $\on{Ber}(1/2)-1/2$. For $j\in\{1,2\}$, let $\mc{E}_j$ be the event 
$$\mc{E}_j =  \bigcap_{i=0}^{t-2}\{|w^{T}R^{i}_{n-1}u_j| \le C K^{i}(\log{2t})n^{(i+1)/2}\}\cap \left\{\bigg|w^T\bigg(\sum_{i\ge t-1}(R_{n-1}/L)^i\bigg)u_j\bigg|\le\frac{CK^{t-1}n^{t/2}}{2L^{t-1}}\right\}. $$
On the event $\{R_{n-1} \in \mc{G}_K\}$, we have $\snorm{R_{n-1}^{i}} \le (K\sqrt{n})^{i}$, so that $\snorm{R_{n-1}^{i}}_{\on{HS}} \le K^{i}n^{(i+1)/2}$. 
Therefore, on this event, for any $L \ge 2K\sqrt{n}$, we have 
\[\norm{\sum_{i\ge t-1}(R_{n-1}/L)^i}_{\on{HS}}\le\frac{2K^{t-1}n^{t/2}}{L^{t-1}}.\]

\textbf{Step 3: }For independent random vectors $u,w$ as above, and a fixed $(n-1)\times (n-1)$ matrix $Q$, consider the random quadratic polynomial $w^{T} Q u$. It follows from standard hypercontractive estimates (cf. \cite[Theorem~10.24]{OD14}), that there exists an absolute constant $c > 0$ such that for all $x \ge 0$,
\[\mb{P}[|w^TQu|\ge x\snorm{Q}_{\on{HS}}]\le c^{-1}\exp(-cx).\]
From this, the Hilbert-Schmidt norm bounds in the previous step, and the union bound, it follows that for $j \in \{1,2\}$, and for any $R_{n-1} \in \mc{G}_K$, we have
$$\mb{P}_{u_j, w}[\mc{E}_j(u_j, w, R_{n-1})] \ge \frac{1}{2},$$
provided that $C > 0$ is chosen sufficiently large. 

Fix $R_{n-1} \in \mc{G}_K$. For $j \in \{1,2\}$, let $W_j(R_{n-1})$ be the set of those vectors $w \in \{-1,0,1\}^{n-1}$ for which 
$$\mb{P}_{u_j}[\mc{E}_j(u_j, w, R_{n-1})] \ge \frac{1}{4}.$$
Note that since $u_1$ and $u_2$ are identically distributed, $W_1(R_{n-1}) = W_2(R_{n-1}) =: W(R_{n-1})$. Then, by combining the conclusion of the previous step with averaging (reverse Markov's inequality), it follows that
$$\mb{P}[w \in W(R_{n-1})] \ge \frac{1}{4}.$$

\textbf{Step 4: }Fix $R_{n-1} \in \mc{G}_K$ and $w \in W(R_{n-1})$. Then, we have
\begin{align*}
    \mb{P}_{u}[\mc{E}(u,w,R_{n-1})]
    &\ge \mb{P}_{u_1, u_2}[\mc{E}(u_1 - u_2, w, R_{n-1})\cap \mc{E}_1(u_1, w, R_{n-1}) \cap \mc{E}_2(u_2, w, R_{n-1})]\\
    &= \mb{P}_{u_1, u_2}[\mc{E}(u_1 - u_2, w, R_{n-1}) \mid \mc{E}_1(u_1, w, R_{n-1}) \cap \mc{E}_2(u_2, w, R_{n-1})] \times\\
    & \quad \times \mb{P}_{u_1}[\mc{E}_1(u_1, w, R_{n-1})]\mb{P}_{u_2}[\mc{E}(u_2, w, R_{n-1})]\\
    &\ge \frac{1}{16}\mb{P}_{u_1, u_2}[\mc{E}(u_1 - u_2, w, R_{n-1}) \mid \mc{E}_1(u_1, w, R_{n-1}) \cap \mc{E}_2(u_2, w, R_{n-1})],
\end{align*}
where the final inequality uses that $w \in W(R_{n-1})$. 

Let us now bound from below the very last term. Conditioned on the event $\mc{E}_{1}(u_1, w, R_{n-1}) \cap \mc{E}_{2}(u_2, w, R_{n-1})$, we know that for $j \in \{1,2\}$, $$\bigg|w^T\bigg(\sum_{i\ge t-1}(R_{n-1}/L)^i\bigg)u_j\bigg|\le\frac{CK^{t-1}n^{t/2}}{2L^{t-1}},$$
so that by the triangle inequality,
$$\bigg|w^T\bigg(\sum_{i\ge t-1}(R_{n-1}/L)^i\bigg)u\bigg|\le\frac{CK^{t-1}n^{t/2}}{2L^{t-1}}.$$
Moreover, conditioned on the event $\mc{E}_{1}(u_1, w, R_{n-1}) \cap \mc{E}_2(u_2, w, R_{n-1})$, we know that the $(t-1)$-dimensional vectors
$$(w^{T} R_{i}^{n-1} u_j)_{0\le i \le t-2}$$
are i.i.d for $j \in \{1,2\}$, have each coordinate equal to a half-integer, and lie in a $(t-1)$-dimensional region with total number of points in $(\mb{Z}/2)^{(t-1)}$ at most
\[4^{t-1}C^{t-1}K^{\frac{(t-1)(t-2)}{2}}(\log 2t)^{t-1}n^{\frac{t(t-1)}{4}}.\]
In particular, by Cauchy-Schwarz, we see that these two $(t-1)$-dimensional vectors coincide (conditioned on $\mc{E}_1(u_1, w, R_{n-1})\cap \mc{E}_2 (u_2, w, R_{n-1})$ is at least
$$(4^{t-1}C^{t-1}K^{\frac{(t-1)(t-2)}{2}}(\log 2t)^{t-1}n^{\frac{t(t-1)}{4}})^{-1}.$$
But whenever this happens, we must also have
$$\bigcap_{i=0}^{t-2}\{w^TR_{n-1}^iu = 0\}.$$
To summarize, we have shown that for any $R_{n-1} \in \mc{G}_K$ and for any $w \in W(R_{n-1})$,
$$\mb{P}_{u}[\mc{E}(u,w,R_{n-1})] \ge 4^{-t-3}C^{1-t}K^{-\frac{(t-1)(t-2)}{2}}(\log 2t)^{1-t}n^{-\frac{t(t-1)}{4}}.$$

\textbf{Step 5: }The desired lower bound now follows easily. We have
\begin{align*}
    \mb{P}_{u,w,R_{n-1}}[\mc{E}\cap \{R_{n-1} \in \mc{G}_K\}]
    &\ge \mb{P}_{u,w}[\mc{E}(u,w,R_{n-1})|R_{n-1} \in \mc{G}_K](1-2^{-n})\\
    &\ge \mb{P}_{u,w}[\mc{E}(u,w,R_{n-1})\cap \{w \in W(R_{n-1})\} | R_{n-1} \in \mc{G}_K](1-2^{-n})\\
    &\ge \mb{P}_{u}[\mc{E}(u,w,R_{n-1})| \{w \in W(R_{n-1})\} \cap \{R_{n-1} \in \mc{G}_K\} ]\times\\
    &\quad \times \mb{P}_{w}[w\in W(R_{n-1}) | R_{n-1} \in \mc{G}_K](1-2^{-n})\\
    &\ge \frac{(1-2^{-n})}{4}\mb{P}_{u}[\mc{E}(u,w,R_{n-1})| \{w \in W(R_{n-1})\} \cap \{R_{n-1} \in \mc{G}_K\} ]\\
    &\ge \frac{(1-2^{-n})}{4}4^{-t-3}C^{1-t}K^{-\frac{(t-1)(t-2)}{2}}(\log 2t)^{1-t}n^{-\frac{t(t-1)}{4}},
\end{align*}
where the final inequality follows from the end of Step 4. Combining this with Step 1 completes the proof. 
\end{proof}

\bibliographystyle{amsplain0.bst}
\bibliography{main.bib}

\end{document}